\documentclass[oneside,reqno]{amsart}
\usepackage{amssymb,latexsym,cases}
\usepackage{paralist}
\usepackage{amssymb} %% added by the authors
\usepackage{graphics} %% add this and next lines if pictures should be in esp format
\usepackage{graphicx,psfrag}
\usepackage{epsfig} %For pictures: screened artwork should be set up with an 85 or 100 line screen
\usepackage{float}
\usepackage[colorlinks=true]{hyperref}
\usepackage{graphicx,psfrag}
\numberwithin{equation}{section}

\newtheorem{lem}{Lemma}[section]
\newtheorem{thm}{Theorem}[section]

\newtheorem{obs}{Remark}[section]
\newcommand{\kk}{\kappa}
\newcommand{\dy}{\partial}
\newcommand{\ddt}[1]{\frac{\mathrm{d}{#1}}{\mathrm{d}{t}}}

\newcommand{\Zahl}{\mathbb{Z}}
\newcommand{\Real}{\mathbb{R}}

\newcommand{\ex}{\mathrm{e}}

\newcommand{\gb}{\nabla}

\newcommand{\aand}{\quad\textrm{and}\quad}

\newcommand{\tht}{\theta}

\newcommand{\ub}{u}

\newcommand{\kapz}{\kappa_0}

\newcommand{\kapu}{\bar\kappa}
\newcommand{\kapc}{\kappa_\tht}

\newcommand{\flux}{\Theta}

\newcommand{\avg}[1]{\langle{#1}\rangle}

\newcommand\lgl{\langle}
\newcommand\rgl{\rangle}
\newcommand\Lim{\mathop{\hbox{Lim}}}

\renewcommand\Pr{\mathrm{Sc}}
\newcommand\Sc{\mathrm{Sc}}
\newcommand\Q{g}

\newcommand\kz{\kappa_0}
\newcommand\kfup{\bar\kappa}
\newcommand\kfdn{\vphantom{\kz}\smash{\underset{\textstyle\bar{}}\kappa}}
\newcommand\kg{\kappa_g}
\newcommand\kgup{\kg}

\newcommand\keps{\kappa_\epsilon}
\newcommand\keta{\kappa_\eta}
\newcommand\ktheta{\kappa_\theta}
\newcommand\ktau{\kappa_\tau}
\newcommand\ksig{\kappa_\sigma}
\newcommand\kbeta{\kappa_\text{2D}}
\newcommand\kbetap{\kappa_\text{3D}}
\newcommand\kidn{{\underset{\textstyle\bar{}}\kappa}\vphantom{\kappa}_{\mathrm{i}}^{}}

\newcommand\cE{\mathcal{E}}
\newcommand\cP{\mathcal{P}}
\newcommand\cT{\mathcal{T}}
\newcommand\bfe{\mathsf{e}}
\newcommand\bfE{\mathsf{E}}
\newcommand\bft{\mathsf{\vartheta}}

\newcommand\me{{\mathfrak e}}
\newcommand\mek{{\mathfrak e}_\kappa}
\newcommand\mE{{\mathfrak E}}
\newcommand\mEk{{\mathfrak E}_\kappa}

\newcommand\uh{\hat u}

\begin{document}
\title[Passive tracer cascade]
{Energy spectra and passive tracer cascades\\in turbulent flows}
\author{M. S. Jolly$^{1,\dagger}$}
\address{$^1$Department of Mathematics\\
Indiana University\\ Bloomington, IN 47405}
\address{$\dagger$ corresponding author}
\author{D. Wirosoetisno$^{2}$}
\address{$^2$Department of Mathematics Sciences\\
Durham University\\ Durham U.K.\ \ DH1 3LE}
\email[M. S. Jolly]{msjolly@indiana.edu}
\email[D. Wirosoetisno]{djoko.wirosoetisno@durham.ac.uk}

\thanks{This work was supported in part by NSF grant number DMS-1418911 and Leverhulme Trust grant VP1-2015-036}

\date{\today}

\subjclass[2010]{
35Q30, %
76F02, % turbulence: fundamentals -- 76F05 is better fit?
76F25} % turbulence: transport, mixing
\keywords{Navier-Stokes equations, turbulence, enstrophy cascade}
\begin{abstract}
We study the influence of the energy spectrum on the extent of the
cascade range of a passive tracer in turbulent flows. The interesting
cases are when there are two different spectra over the potential
range of the tracer cascade (in 2D when the tracer forcing is in the
inverse energy cascade range,
and in 3D when the Schmidt number Sc is large). The extent of the tracer
cascade range is then limited by the width of the range for the
shallower of the two energy spectra. Nevertheless, we show that in
dimension $d=2,3$ the tracer cascade range extends (up to a logarithm)
to $\kappa_{d\text{D}}^{p}$, where $\kappa_{d\text{D}}$ is the
wavenumber beyond which diffusion should dominate and $p$ is arbitrarily
close to 1, provided Sc is larger than a certain power (depending on
$p$) of the Grashof number. We also derive estimates which suggest that
in 2D, for Sc${}\sim1$ a wide tracer cascade can coexist with a
significant inverse energy cascade at Grashof numbers large enough to
produce a turbulent flow.
\end{abstract}
\maketitle

%\section{R\'esum\'e}
\vspace{\baselineskip}

\normalsize

% ===========================================================================

\section{Introduction}

Passive tracers play an important role in the study of fluid motion.
On the one hand, experimental and observational studies of fluid flows rely heavily on passive tracers to deduce the advecting velocity field.
On the other hand, knowledge of the underlying fluid flows is essential to predict the future dispersion of tracers (particularly, but not exclusively, harmful ones).

It is natural to believe that if the advecting fluid flow is turbulent (however this is defined), the evolution of the tracer will be turbulent as well.
Following the pioneering work by Kolmogorov, Obukhov \cite{obukhov:49} and Corrsin \cite{corrsin:51} argued that if the energy spectrum of the fluid is $\cE(\kappa)=K\kappa^{-n}$, a passive tracer whose dissipation rate is $\chi$ should have the spectrum $\cT(\kappa)\sim\chi K^{-1/2}\kappa^{(n-5)/2}$ between the injection and dissipation scales.
Thus, in the inertial range in 3D, both the energy and tracer spectra scale as $\kappa^{-5/3}$.
Following Kraichnan \cite{Kr71}, in the direct enstrophy cascade range in 2D, the energy spectrum should scale as $\kappa^{-3}$, giving a $\kappa^{-1}$ tracer spectrum.
Although these scaling arguments were derived with little reference to the governing equations, they have been supported to a surprising extent by experimental and numerical works (cf.~\cite{davidson:t,Vallis}), primarily in 3D, slightly less so for 2D and still less so for tracers.

In 3D and 2D, respectively, dissipative effects are expected to dominate beyond the Kolmogorov and Kraichnan wavenumbers $\keps$ and $\keta$.
The corresponding scales for our tracer depend in addition on the Schmidt number $\Sc$, i.e.\ the ratio of the viscosity to the tracer dissipativity.
Another lengthscale of great importance is the Taylor microscale.
Initially (and to this day among experimentalists) defined using the velocity correlation, mathematicians prefer to use an alternate definition for $\ktau$ in terms of the energy and its dissipation rate \eqref{q:ktausigtht};
the two definitions can be shown to be (nearly) equivalent under some assumptions.
Assuming that $\ktau$ is much greater than the forcing scale, it has been proved rigorously that a direct energy cascade exists for solutions of 3D NSE \cite{FMRT}.
Similarly, in 2D one defines in terms of the enstrophy and its dissipation rate a wavenumber $\ksig$, which if sufficiently larger than the forcing scale rigorously implies the existence of direct enstrophy cascade \cite{FJMR}.
In section~\ref{s:i-tht}, we derive an analogous result for tracers in terms of a corresponding wavenumber $\ktheta$.

% Even though it may appear superficially
While it is plausible that $\ktau$, $\ksig$ and $\ktheta$ are large for turbulent solutions of the NSE and the advected tracers, these remain unproved (directly from the NS and the tracer equations) to this day.
If one were to assume the expected spectra, namely $\epsilon^{2/3}\kappa^{-5/3}$ and $\eta^{2/3}\kappa^{-3}$, however, it has been shown that $\ktau\sim\keps^{2/3}\kz^{1/3}$ in 3D \cite{DFJ5} and $\ksig\sim\keta$ up to a logarithm in 2D \cite{DFJ3}.
Following this approach, we prove the tracer analogues in sections \ref{2DTracer} and~\ref{3DTracer}.
There are a number of qualitatively distinct cases here, depending on the viscosity $\nu$ and tracer dissipativity $\mu$, as well as on the injection scales of energy $\kappa_f$ and of the tracer $\kg$.
When $\nu\gg\mu$, it is possible for $\ktheta$ to asymptotically approach (up to constants and logarithms) its largest possible value, in that $\ktheta\sim\kappa_{d\textrm{D}}^{1-r}\kz^r$ for any $r<1$, both for $d=2$ (\S\ref{2DTracer}) and $d=3$ (\S\ref{3DTracer}) in a periodic domain of volume $(2\pi/\kz)^d$.
When $\nu\sim\mu$, the situation is more complicated as discussed in detail below.

The rest of this paper is structured as follows.
After some mathematical setups in Section~\ref{prelim},
we recall the heuristic argument for the tracer spectra in Section~\ref{connect}.
Earlier NSE estimates for the enstrophy and energy transfer rates in terms of $\ksig$ and $\ktau$ in 2D and 3D are gathered in Section \ref{s:gen}, along with the implications that the expected energy spectra have on these wavenumbers, vis-\`a-vis $\keta$, $\keps$, respectively.
An analogous estimate for the tracer transfer rate in terms of $\ktheta$ is also derived in Section~\ref{s:gen}.
We treat 2D tracer flow in Section~\ref{2DTracer} and 3D tracer flow in Section~\ref{3DTracer}.
% \mnote{Djoko: I think it suffices to have the simpler 2D case, i.e., large Schmidt number, first, rather than switch 3D and 2D. This way, the tricky 2D stuff with the figures is postponed somewhat.  I have reordered things in Section 5, accordingly, and made minor adjustments in the proofs to acommodate.}

% ===========================================================================

\section{Preliminaries}\label{prelim}

We consider the evolution of a passive scalar $\theta$ under
a prescribed velocity field $u(x,t)$ and a time-independent source $\Q=\Q(x)$,
\begin{equation}\label{q:dcdt}\begin{aligned}
   \dy_t\tht - \mu\Delta\tht+ u\cdot\gb\tht  &= \Q\\
   \int_\Omega \theta \ dx = 0\;,  \qquad \int_\Omega g \ dx &= 0  \;,
\end{aligned}\end{equation}
with periodic boundary conditions in $\Omega =[0,L]^d$ for $d=2,3$.
We focus on the case where $u$ satisfies
 the incompressible Navier-Stokes equations
\begin{equation}\label{NSEphys}\begin{aligned}
   &\dy_t u - \nu \Delta u  + (u\cdot\nabla)u + \nabla p = F(t), \\
   &\nabla \cdot u = 0 \\
   &\int_{\Omega} u\,  dx =0 \;,\qquad \int_{\Omega} F\,  dx =0 \\
   &u(x,t_0) = u_0(x) \;.
\end{aligned}\end{equation}
We write \eqref{NSEphys} as
a differential equation in a certain Hilbert space $H$ (see \cite{CF88,T97}),
\begin{equation}\label{NSE}\begin{aligned}
   &\frac{d}{dt}u(t) + \nu Au(t) + B(u(t),u(t)) = f(t),\\
   &u(t) \in H,\quad
	t \ge t_0\quad
	\aand u(t_0)=u_0\,.
\end{aligned}\end{equation}
The phase space $H$ is the closure in $L^2(\Omega)^d$
of all  $\mathbb{R}^2$-valued trigonometric poly\-nomials $u$ such that
\[
   \nabla \cdot u = 0 \qquad \text{ and} \quad \int_{\Omega} u(x) \;dx =0.
\]
The bilinear operator $B$ is defined as
\[
   B(u,v)=\cP\left( (u \cdot \nabla) v \right),
\]
where $\cP$ is the Helmholtz--Leray orthogonal projector of $L^2(\Omega)^d$ onto $H$ and $f=\cP F$.
The scalar product in $H$ is taken to be
\[
   (u,v)=\int_{\Omega}u(x)\cdot v(x) \;dx,
%\quad \text{where} \quad a\cdot b=a_1b_1+a_2b_2,
\]
with the associated norm
\begin{equation}\label{L2norm}
   |u|=(u,u)^{1/2}=\Bigl(\int_{\Omega}u(x)\cdot u(x) \;dx\Bigr)^{1/2}.
\end{equation}
The operator $A=-\Delta$ is self-adjoint, and its eigenvalues are of the form
\[
   ({2\pi}/{L})^2 k\cdot k
	\quad \text{where }k \in \Zahl^d \setminus \{ 0\}.
\]
We denote these eigenvalues by
\[
   0<\lambda _0=\left(2\pi/ L\right)^2 \le \lambda _1\leq\lambda_2 \le \cdots
\]
arranged in non-decreasing order (counting multiplicities)
and write $w_0$, $w_1$, $w_2,\cdots$, for the corresponding normalized
eigenvectors (i.e.\ $|w_j|=1$ and $Aw_j=\lambda_j w_j$ for $j=0,1,2,\cdots$).

For $\alpha\in\Real$, the positive roots of $A$ are defined by linearity from
\[
   A^{\alpha}w_j= \lambda_j^{\alpha} w_j, \quad \text{for} \ j=0,1,2, \cdots
\]
on the domain
\[
   D(A^{\alpha})=\Bigl\{u \in H: \sum_{j=0}^{\infty} \lambda_j^{2\alpha}(u,w_j)^2 < \infty\Bigr\}.
\]
We take the natural norm on $V=D(A^{1/2})$ to be
\begin{equation}\label{Vnorm}
\|u\|=|A^{1/2}u|=\left(\int_{\Omega}\sum_{j=1}^{d}
\frac{\partial}{\partial x_j}u(x)\cdot
              \frac{\partial}{\partial x_j}u(x)  \;dx\right)^{1/2}
=\left(\sum_{j=0}^{\infty} \lambda_j(u,w_j)^2 \right)^2.
\end{equation}
Since the boundary conditions are periodic,
we may express an element in $H$ as a Fourier series
\begin{equation}\label{Fourier}
u(x)=\sum_{k\in \mathbb{Z}^d}\uh_k e^{i\kappa_0 k\cdot x}\;,
\end{equation}
where
\begin{equation}\label{k0def}
\kappa_0=\lambda_0^{1/2}=\frac{2\pi}{L}, \quad \uh_0=0, \quad
\uh_k^*=\uh_{-k}
\end{equation}
and due to incompressibility,
$k \cdot \uh_k=0$.
We associate to each term in  \eqref{Fourier} a {\it wavenumber} $\kappa_0|k|$.
Parseval's identity reads as
$$
|u|^2=L^d\sum_{k\in \Zahl^d} \uh_k\cdot \uh_{-k} = L^d\sum_{k\in \Zahl^d} |\uh_k|^2.
$$
% (it should be clear from context when $|\cdot|$ denotes
% the length of a vector in $\Cplx^d$).

Two important dimensionless parameters are the Grashof and Schmidt numbers,
\begin{equation}\label{Grashof}
   G := \frac{|f|}{\nu^2\kz^{3-d/2}}
   \aand
   \Sc := \frac{\nu}{\mu}.
\end{equation}
The former indicates the complexity of the (velocity) flow, and the latter the importance of (momentum) viscosity relative to tracer dissipativity.

We define the average {\it energy\/}, {\it enstrophy\/} and {\it tracer variance\/}
\begin{equation}\label{quantities1}
   \bfe=\frac{1}{L^d}\lgl |u|^2\rgl\,, \qquad
   \bfE=\frac{1}{L^d}\lgl \|u\|^2\rgl \aand
   \frac{1}{L^d}\lgl |\tht|^2\rgl,
\end{equation}
as well as their dissipation (diffusion) rates
\begin{equation}\label{quantities2}
   \epsilon := \frac{\nu}{L^d}\lgl \|u\|^2\rgl\,, \qquad
   \eta := \frac{\nu}{L^d}\lgl |Au|^2\rgl \aand
   \chi := \frac{\mu}{L^d}\lgl |\nabla\tht|^2\rgl.
\end{equation}
%In turbulence theory, to be consistent with \eqref{energyeq},
%\eqref{enstrophyeq}, the actual
%mean energy is $\bfe/2$,  and the actual enstrophy dissipation
%rate is $2\eta$.  To simplify the presentation we omit these factors
%of 2 in \eqref{quantities1},   \eqref{quantities2}.
By the classical dimensional arguments, the dissipation range is expected to start at
\begin{equation}\label{ketadef}
   \keps= \left(\frac{\epsilon}{\nu^3}\right)^{1/4}
   \aand
   \kappa_{\eta}=\left(\frac{\eta}{\nu^3}\right)^{1/6},
\end{equation}
in 3D and 2D respectively;
these are sometimes known as the Kolmogorov and Kraichnan wavenumbers.
Their analogues for the tracer cascade are more complicated and depend on the advecting velocity; see $\kbeta$ and $\kbetap$ in \S\ref{2DTracer}--\S\ref{3DTracer} below.
Another set of important wavenumbers are
\begin{equation}\label{q:ktausigtht}
   \ktau^2 := \frac{\lgl\|u\|^2\rgl}{\lgl|u|^2\rgl},\quad
   \ksig^2 := \frac{\lgl|\Delta u|^2\rgl}{\lgl\|u\|^2\rgl} \aand
   \ktheta^2 := \frac{\lgl\|\theta\|^2\rgl}{\lgl|\theta|^2\rgl}.
\end{equation}
In 3D turbulence, $\ktau$ is closely related to the Taylor wavenumber, the scale at which the velocity correlation is lost; it has been shown that direct energy cascade takes place within the range $(\kfup,\ktau)$.
Its analogues in 2D and tracer turbulence are $\ksig$ and $\ktheta$, with corresponding results on enstrophy \cite{FJMR} and tracer [\eqref{directenergy-theta} below] cascades.

% We assume that the scales where the viscous term of the NSE dominates are well beyond those of the force.

\medskip
We make use of the following notation:
$a\lesssim b$ means $a\le c\,b$ for a nondimensional universal constant $c$, independent of $G$ and $\Pr$, {\it under the condition that} $G \ge G_*$ where $G_*$ may be different for each inequality, and similarly for $\gtrsim$\,.
By $a\sim b$ we mean that both $a\lesssim b$ and $b\lesssim a$ hold.
% For convenience, we will in some instances write
%\begin{equation}\label{leLdef}
%a \leL b  \quad  \text{when} \quad a \le C(\log (s\bG))^{\alpha}b
%\quad \text{for some}\ \alpha\in\bR\;, \ \text{and large enough} \ \bG
%\end{equation}
%and where $C$ and $s$
%are shape factors, with a similar convention for $\geL$.
We write
$a\ll b$ if $a/b < \delta$  for some small $\delta \in (0,1)$, and
$a/b$ is nondimensional provided the ranges of $a, b$ are a priori specified (e.g., for large values of $a, b$).
The value of $\delta$ shall remain unspecified, and may vary from one
statement involving $\ll$ to the next.

Since the infinite time limit is not known to exist, for each solution $u(t)$ of the 2D NSE (Leray--Hopf weak solution in the 3D case) we work with the average
\[
   \lgl\Phi\rgl = \Lim\limits_{T\to\infty}\frac{1}{T}\int_0^T\Phi(u(t))\,dt
   \quad\mbox{for any}\ \Phi\ \mbox{weakly continuous in }H,
\]
where $\Lim$ is a Hahn--Banach extension of the classical time limit.
The average $\lgl\cdot\rgl$ is the mathematical equivalent
of the ensemble average in the statistical theory of turbulence.
See \cite{FJMR,FMRT} for more details.
% \mnote{DJOKO:  A point to finesse here is that we also consider $\lgl |\theta| \rgl$,  $\lgl |\nabla\theta| \rgl$.  This might be explained in terms of $\Phi(u)=|\theta|$  being continuous in $u$. I have tried to avoid here referring to the invariant measure $\mu_u$, in part to avoid notation confusion with diffusivity coefficient $\mu$}

% ===========================================================================

\section{Influence of energy spectrum}\label{connect}

\subsection{Classical theory}

We recall briefly from \cite[Ch.~8]{Vallis} some elements of the Kolmogorov--Obukhov theory for 3D turbulence in a form suitable for its extension to passive tracers.
Suppose that a parcel (``eddy'') of size $1/\kk$ has velocity $U_\kk\sim [\kk\cE(\kk)]^{1/2}$.
Assuming that such an eddy breaks up in the time $\tau_\kk$ it takes to travel its own size, i.e.
\begin{equation}\label{totem}
   \tau_\kk U_\kk=1/\kk
   \quad \text{so that} \quad
   \tau_\kk\sim[\kk^3\cE(\kk)]^{-1/2},
\end{equation}
the resulting downscale energy transfer rate is
\begin{equation}\label{spec1}
  \frac{U_\kk^2}{\tau_\kk} \sim \frac{\kk\cE(\kk)}{\tau_\kk}.
\end{equation}
Assuming that this transfer rate is a constant $\epsilon$ for $\kk$ in the so-called inertial range and solving for $\cE$, we arrive at the Kolmogorov spectrum,
\begin{equation}\label{spec3D}
   \cE_{\text{3D}}(\kk)\sim \epsilon^{2/3}\kk^{-5/3}.
\end{equation}

The situation in 2D is more complicated in that, for scales smaller than the forcing, we expect the enstrophy to undergo a direct cascade to smaller scales, while energy is mainly transferred to larger scales in an inverse cascade for scales larger than the forcing.
Yet a similar dimensional argument in the enstrophy inertial range leads to the Kraichnan spectrum
\begin{equation}\label{spec2D}
   \cE_{\text{2D}}(\kk)\sim \eta^{2/3}\kk^{-3}\;.
\end{equation}

An analogous cascade mechanism for the tracer suggests a connection between its spectrum $\cT(\kk)$ and the energy spectrum.
Taking the amount of tracer (variance) at wavenumber $\kk$ to be $\kk\cT(\kk)$, assuming that it is transferred to wavenumber $2\kk$ by the advecting velocity over a time $\tau_\kk$ given by \eqref{totem}, and setting the transfer rate to a constant $\chi$, we find
\begin{equation}\label{spec3}
  \chi \sim \frac{\kk\cT(\kk)}{\tau_\kk}.
\end{equation}
If we take $\cE(\kk) \sim K \kk^{-n}$ in \eqref{totem} and solve for $\cT$ in \eqref{spec3}, we have
\begin{equation}\label{specT}
   \cT(\kk)\sim \chi K^{-1/2} \kk^{(n-5)/2}.
\end{equation}

\subsection{Mathematical formulation}

These spectral relations can be reformulated in terms of partial sums
\begin{equation}\label{partialsum}
   \bfe_{\kk,2\kk} := \frac{1}{L^d}\sum_{\kk \le \kz |k| < 2\kk} \lgl |\uh_k|^2  \rgl
   \aand
   \bft_{\kk,2\kk} := \frac{1}{L^d}\sum_{\kk \le \kz |k| < 2\kk} \lgl  |\hat{\theta}_k|^2 \rgl.
\end{equation}
% where in each case the summation is over wave vectors satisfying $\kk \le \kz |k| < 2\kk$.
As $L$ increases (so $\kz$ decreases), each quantity in \eqref{partialsum} can be viewed as a Riemann sum approximation of the integral of the corresponding spectrum (this assumes smoothness of the summands, but below we will use this approximation only for explicit functions of $\kk$).
For instance, for the energy in 3D, we have
\[
   \int_{\kk}^{2\kk} \cE_{\text{3D}}(\tilde{\kk}) \;d\tilde{\kk}
	\sim \int_{\kk}^{2\kk} \epsilon^{2/3} \tilde{\kk}^{-5/3} \;d\tilde{\kk}
	= \frac{3}{2} \epsilon^{2/3} \bigl(1-2^{-2/3}\bigr)\kk^{-2/3}
	\sim \epsilon^{2/3} \kk^{-2/3}.
\]
This leads to the energy power law
\begin{equation} \label{power23}
  \bfe_{\kk,2\kk}\sim \epsilon^{2/3} \kk^{-2/3} \quad \text{in 3D}
\end{equation}
and similarly
\begin{equation} \label{power2}
  \bfe_{\kk,2\kk}\sim \eta^{2/3} \kk^{-2} \quad \text{in 2D}.
\end{equation}
We gather the expected spectra according to classical theory in Table~\ref{table1}.

\begin{table}\begin{center}
\begin{tabular}{|c|c||c|c|c|c|c|} \hline
dir. & $d\vphantom{\Big|}$&  $\cE(\kk)$ &  $\bfe_\kk$    & $\cT(\kk)$ & $\bft_\kk$\\ \hline \hline
fwd &  3 \vphantom{\Big|}&	$\epsilon^{2/3} \kk^{-5/3}$ & $\epsilon^{2/3} \kk^{-2/3}$ &	$\chi\epsilon^{-1/3} \kk^{-5/3}$ & $\chi\epsilon^{-1/3} \kk^{-2/3}$\\  \hline
fwd &  2 \vphantom{\Big|}&	$\eta^{2/3} \kk^{-3}$ & $\eta^{2/3} \kk^{-2}$ &	$\chi\eta^{-1/3} \kk^{-1}$ & $\chi\eta^{-1/3}$\\  \hline
bkwd &  2 \vphantom{\Big|}&	$\epsilon^{2/3} \kk^{-5/3}$ & $\epsilon^{2/3} \kk^{-2/3}$ &	$\chi\epsilon^{-1/3} \kk^{-5/3}$ & $\chi\epsilon^{-1/3}\kk^{-2/3}$\\  \hline
\end{tabular}
\end{center}\caption{Spectra according to classical theory}\label{table1}
\end{table}

We conclude this section with a brief calculation regarding the summation of the tracer variance over the relevant wavenumber range assuming that a certain power law holds.  It will be used repeatedly.

\begin{lem}\label{sumlem}
Suppose $\bft_{\kk,2\kk} \sim \alpha\kk^{-p}$ for $\kk_1 \le \kk \le \kk_2$, with $4\kk_1 \le \kk_2$ and $p\ge 0$.
Then
\begin{subnumcases}
   {\bft_{\kk1,\kk2} \sim}
	\alpha\left(\kk_1^{-p}-\kk_2^{-p}\right)  &if $p>0$, \label{sum1}\\
	\alpha\ln(\kk_2/\kk_1) &if $p=0$. \label{sum2}
\end{subnumcases}
\end{lem}

\begin{proof}
As in \cite{DFJ3,DFJ5}, let $J=\lfloor \log_2(\kk_2/\kk_1) \rfloor -1$.
If $p>0$, then
\begin{align*}
\bft_{\kk_1,\kk_2}\sim \sum_{\kk=2^j\kk_1,\,j=0}^J \bft_{\kk,2\kk} \sim
 \frac{\alpha}{\kk_1^{p}}\sum_{j=0}^J (2^{p})^{-j}  &=
  \frac{\alpha}{\kk_1^{p}} \frac{1}{1-2^{-p}}\left[1-(2^{-J})^{-p}\right] \\
  &\sim \frac{\alpha}{\kk_1^{p}} \left[1-\left(\frac{\kk_1}{\kk_2}\right)^p\right].
\end{align*}
If $p=0$,
\[
  \bft_{\kk1,\kk2}\sim  \alpha\sum_{j=0}^J 1= \alpha \log_2(\kk_2/\kk_1) \sim \alpha \ln (\kk_2/\kk_1).
\]\end{proof}

% ===========================================================================

\section{Indicators for cascades}\label{s:gen}

Returning to the Navier--Stokes \eqref{NSE} and tracer equations \eqref{q:dcdt},
we henceforth assume that the forcing functions $F$ and $\Q$ are spectrally-bounded, i.e.\ there exist $\kz<\kgup<\infty$ and $\kz\le\kfdn\le\kfup<\infty$ such that
% \mnote{Mike: I removed $\bar\kappa_g$; if $\kfdn=\kfup$, then $f=0$ ...}
\begin{equation}
   \Q = \Q_{\kz,\kgup}^{} \aand f=f_{\kfdn,\kfup}.
\end{equation}
Given a fixed $\kappa$, we define
\begin{equation}
   u^<:= u _{\kapz,\kappa}\,,   \quad  u^> := u_{\kappa,\infty}
   \aand
   \tht^< := \tht_{\kapz,\kappa} \,, \quad  \tht^> := \tht_{\kappa,\infty}.
\end{equation}

\subsection{Navier--Stokes equations}
We start by giving sufficient conditions for enstrophy and energy cascades.
In terms of the solution of the 2D NSE, the {\it net rate of enstrophy transfer (flux)} is given by $\mE_\kappa= \mE_\kappa^{\rightarrow} - \mE_\kappa^{\leftarrow}$
where
\[
   \mE^{\rightarrow}_\kappa(u) = -\frac{1}{L^2}(B(u^<,u^<),Au^>)
   \aand
   \mE^{\leftarrow}_\kappa(u) = -\frac{1}{L^2}(B(u^>,u^>),Au^<).
\]
are the {\it rates of enstrophy transfer} (low to high) and (high to low), respectively.
It was shown in \cite{FJMR} that
\begin{equation}\label{cascrel}
   1 - \Bigl(\frac{\kappa}{\ksig}\Bigr)^2
	\le \frac{\lgl \mEk \rgl}{\eta}  \le 1
	\quad\text{if}\quad \kfup \le \kappa \le\ksig.
\end{equation}
It follows that if
\begin{equation}\label{ksigkf}
   \ksig \gg \kfup,
\end{equation}
then there exists an {\it enstrophy cascade}:
\begin{equation}\label{enstcascrel}
   \lgl \mEk \rgl \approx \eta \quad \text{for} \quad
   \kfup \le \kappa \ll \ksig \,.
\end{equation}

Similarly, in both 2D and 3D (using Leray--Hopf solutions for the latter) the transfer of energy $\me_\kappa =  \me_\kappa^{\rightarrow} - \me_\kappa^{\leftarrow}$ is shown in \cite{FJMR,FMRT} to satisfy
\begin{equation}\label{directenergy}
   1-\Bigl(\frac{\kappa}{\ktau}\Bigr)^2
	\le \frac{\lgl \me_\kappa \rgl}{\epsilon}
	\le 1
	\quad\text{for}\quad
   \kfup \le \kappa \le \ktau,
\end{equation}
where
\[
   \me^{\rightarrow}_\kappa(u) = -\frac{1}{L^d}(B(u^<,u^<),u^>)
   \aand
   \me^{\leftarrow}_\kappa(u) = -\frac{1}{L^d}(B(u^>,u^>),u^<).
\]
Thus if
\begin{equation}\label{ktaukf}
   \ktau \gg \kfup
\end{equation}
there is a direct energy cascade:
\begin{equation}\label{encascrel}
   \lgl \mek \rgl \approx \epsilon \quad \text{for} \quad
   \kfup \le \kappa \ll \ktau \,.
\end{equation}
It is easy to show that $\ktau \le \ksig$, which is consistent with the expectation that for a 2D flow a direct enstrophy cascade be more pronounced than a direct energy cascade.

We note a couple of useful bounds for $\keta$ and $\keps$.
For the 2D NSE (regardless of whether the flow is turbulent), it was shown in \cite{FMT93} that
\begin{equation}\label{FMTest}
   G^{1/6} \lesssim {\keta}/{\kz} \le G^{1/3}.
\end{equation}
While for the 3D NSE, \cite{DFJ5} showed that
% \mnote{Mike: isn't the condition superfluous since we're using $\lesssim$?}
\begin{equation}\label{gen3D}
   \bigl({\kz}/{\kfup}\bigr)^{5/8} G^{1/4} \lesssim \frac{\keps}{\kz}
%	\qquad\text{for all}\quad G \gtrsim \bigl({\kfup}/{\kz}\bigr)^{3/2}.
\end{equation}
As noted in the introduction, at present we do not have independent lower bounds on $\ksig$ and $\ktau$ (other than the trivial, and useless, ones).

\medskip
If one assumes the power spectrum (which {\em a priori\/} says nothing about energy transfer rates), however, one does obtain lower bounds on $\ksig$ and $\keta$, or equivalently by \eqref{ksigkf} and \eqref{ktaukf}, sufficient conditions for the enstrophy and energy cascades.
In 2D we have the following estimate from \cite{DFJ3}.
\begin{thm}\label{2Dconnect}
If for the 2D NSE we have
\begin{equation}\label{law11}
   \bfe_{\kk,2\kk} \sim \eta^{2/3}\kk^{-2}
	\quad \text{for} \quad \kidn \le \kk \le \keta
\end{equation}
with $4\kidn \le \keta$ and
\begin{equation}\label{bulk}
   \lgl \|u_{\kz,\kidn}\|^2\rgl \lesssim \lgl \| u_{\kidn,\infty}\|^2 \rgl \;,
\end{equation}
then
\begin{equation}\label{ksigbig}
   \ksig^2 \sim \keta^2 / \ln(\keta/\kidn).
\end{equation}
\end{thm}

\noindent
The wavenumber $\kidn$ marks the start of the inertial range.
Based on \eqref{cascrel} and \eqref{directenergy}, we expect that $\kidn \sim \kfup$.

Thanks to \eqref{FMTest},
% a bound in \cite{FMT93} for the 2D NSE
% \begin{equation}\label{FMTest}
%    G^{1/6} \lesssim \frac{\keta}{\kz} \le G^{1/3},
% \end{equation}
the dissipation wavenumber $\keta$ can be controlled by the Grashof number.
Thus, under \eqref{law11}, $\ksig$ can indeed be made large by increasing $G$.
% The bounds in \eqref{FMTest} are for a general solution of the 2D NSE
% (regardless of whether the flow is turbulent).
It is shown in \cite{FJM2} that if conversely \eqref{ksigbig} holds, then
one side of the power law holds (up to a log)
% \mnote{Djoko: I think we also did this without finite time averages, but I have yet to track it down.}
\begin{equation}\label{law111}
\bfe_{\kk,2\kk} \lesssim \eta^{2/3}\kk^{-2}\ln({\keta}/{\kidn}) \quad \text{for} \quad \kidn \le \kk \le \keta\;.
\end{equation}
Moreover, under \eqref{ksigbig} it is shown in \cite{DFJ3} that \eqref{FMTest} is sharpened to
\begin{equation}\label {ketasharp}
   \left(\frac{\kz}{\kfup}\right)^{1/4} \frac{G^{1/4}}{(\ln G)^{1/4}}
	\lesssim \frac{\keta}{\kz}
	\lesssim \Bigl(\frac{\kfup}{\kz}\Bigr)^{1/4} G^{1/4}(\ln G)^{1/8}
%	\qquad \forall \ G \gtrsim \Bigl(\frac{\kfup}{\kz}\Bigr)^2 \;.
\end{equation}

The following 3D analogue of Proposition \ref{2Dconnect} is proved in \cite{DFJ5}.
\begin{thm}\label{3Dconnect}
If for a Leray--Hopf solution to the 3D NSE, we have
\begin{equation*}%\label{23skidoo}
   \bfe_{\kk,2\kk} \sim \epsilon^{2/3}\kk^{-2/3} \quad \text{for} \quad \kfup\le \kk \le \keps
\end{equation*}
with $4\kfup \le \keps$ and
\begin{equation*}%\label{bulk}
   \lgl |u|^2\rgl \sim \lgl | u_{\kfup,\keps}|^2 \rgl \;,
\end{equation*}
then
\begin{equation}\label{ktaubig}
   \ktau^3 \sim \keps^2 \kfup\,.
\end{equation}
\end{thm}

% It is also shown in \cite{DFJ5} that, in general,
% \begin{equation}\label{gen3D}
% \left(\frac{\kz}{\kfup}\right)^{1/8} G^{1/4} \lesssim \frac{\keps}{\kz} \;, \qquad \forall \ G \gtrsim \left(\frac{\kfup}{\kz}\right)^{3/2}
% \end{equation}
\noindent
Assuming \eqref{ktaubig}, the bound \eqref{gen3D} can be sharpened to
\begin{equation}\label{sharp3D}
   \left(\frac{\kz}{\kfup}\right)^{11/16} G^{3/8} \lesssim \frac{\keps}{\kz}
	\lesssim \left(\frac{\kz}{\kfup}\right)^{1/8} G^{3/8},
	\qquad \forall \ G  \gtrsim \Bigl(\frac{\kfup}{\kz}\Bigr)^{3/2}.
\end{equation}

The powers in \eqref{ksigbig} and \eqref{ktaubig} are suggestive of
the extent to which the corresponding fluxes are constant over
a given range, or alternatively, the width of the inertial range in each case.

\subsection{Passive tracer}\label{s:i-tht}

A condition for a cascade of the tracer is derived just as those for the NSE.
Let $\kappa$ and $\kgup$ be fixed with $\kappa>\kgup$.
Multiply \eqref{q:dcdt} by $\tht^>$ in $L^2$ to get
\begin{equation}\begin{aligned}
   \frac12\ddt{\;}|\tht^>|^2 + \mu|\gb\tht^>|^2
	&= -(\ub\cdot\gb\tht^<,\tht^>) + (\Q^>,\tht^>)\\
	&= -(\ub^<\cdot\gb\tht^<,\tht^>) + (\ub^>\cdot\gb\tht^>,\tht^<) + (\Q^>,\tht^>)\\
	&= L^d \flux_\kappa + (\Q^>,\tht^>)
\end{aligned}\end{equation}
where
\[
   \flux_\kappa := \frac{1}{L^d} \left[-(\ub^<\cdot\gb\tht^<,\tht^>) + (\ub^>\cdot\gb\tht^>,\tht^<) \right]
\]
is the downscale (i.e.\ towards larger $|k|$) flux of $\tht$ through wavenumber $\kappa$.
Now $\Q^>=0$ since $\kappa>\kgup$, so upon taking average, the time derivative
disappears and we get
\begin{equation}\label{q:ek}
   \mu\avg{|\gb\tht^>|^2} = \avg{\flux_\kappa}.
\end{equation}
%Assuming that the limit exists, we define
%\begin{equation}
%   \avg{\flux} := \lim_{\kappa\to\infty}\,\avg{\flux_\kappa}.
%\end{equation}

The tracer ``energy'' cascade mechanism
%\cite[\S8]{foias-jolly-manley-rosa-temam:05}
requires that $\avg{\flux_\kappa}$ is (nearly) constant
for $\kappa\in[\kappa_*,\kappa^*]\subset[\kapu,\kapc]$.
Noting that
\begin{equation}\label{tracer cascade}
\begin{aligned}
  \chi \ge  \avg{\flux_\kappa} &= \frac{\mu}{L^{d}} \avg{|\gb\tht^>|^2}
	& &= \frac{\mu}{L^{d}}\avg{|\gb\tht|^2} - \frac{\mu}{L^{d}}\avg{|\gb\tht^<|^2}\\
	& = \chi - \kappa^2 \frac{\mu}{L^d}\avg{|\tht^<|^2}
	& &\ge \chi - \kappa^2 \frac{\mu}{L^d}\avg{|\tht|^2}\\
	&= \chi-\frac{\kappa^2}{\kapc^2}\frac{\mu}{L^d}\avg{|\gb\tht|^2}
	& &=\chi \Bigl[1 - \Bigl(\frac{\kappa}{\kapc}\Bigr)^2\Bigr],
\end{aligned}\end{equation}
we obtain the tracer analogue of \eqref{cascrel} and \eqref{directenergy},
\begin{equation}\label{directenergy-theta}
   1-\Bigl(\frac{\kappa}{\ktheta}\Bigr)^2
	\le \frac{\lgl \flux_\kk \rgl}{\chi}
	\le 1
	\quad\text{for}\quad
   \kg \le \kappa \le \ktheta.
%	:= \biggl(\frac{\lgl |\nabla \theta |^2 \rgl}{ \lgl |\theta|^2 \rgl }\biggr)^{1/2}.
\end{equation}

The relations \eqref{cascrel}, \eqref{directenergy} and \eqref{directenergy-theta} all imply cascades (more precisely, constancy of fluxes) provided that the indicator wavenumbers $\ksig$, $\ktau$ and $\ktheta$ are sufficiently large.
Criteria on the forcing $f$ and $g$ that would give these conditions, directly from the NSE without further assumptions, so far remain elusive.
% , conditions whose direct proofs (from the NSE) remain elusive.

% We next compute the useless upper bound
% \begin{equation}
%    \flux_\kappa = (\ub\cdot\gb\tht^>,\tht^<)
% 	\le |\ub|_{L^\infty}^{}|\gb\tht^<|_{L^2}^{}|\tht^>|_{L^2}^{}
% 	\le \frac{1}{\kappa}\,|\ub|_{L^\infty}^{}|\gb\tht^>|_{L^2}^{}|\gb\tht^<|_{L^2}^{}.
% \end{equation}
% Bounding $|\ub|$ for all time, averaging and using \eqref{q:Afg},
% \begin{equation}
%    \avg{\flux_\kappa}
% 	\le \frac1\kappa\,|\ub|_{L_t^\infty L_\xb^\infty}^{} \avg{|\gb\tht^>|\,|\gb\tht^<|}
% 	\le \frac1\kappa\,|\ub|_{L_t^\infty L_\xb^\infty}^{} \avg{|\gb\tht^>|^2}^{1/2}\avg{|\gb\tht^<|^2}^{1/2}
% \end{equation}
% With \eqref{q:ek}, this then implies
% \begin{equation}
%    \avg{|\gb\tht^>|^2} \le \frac1{\mu^2\kappa^2}\,|\ub|_{L_t^\infty L_\xb^\infty}^2 \avg{|\gb\tht^<|^2}.
% \end{equation}

% ===========================================================================

\section{2D Case, effect of energy spectrum on $\ktheta$}\label{2DTracer}

In this section, we prove tracer analogues of Proposition \ref{2Dconnect},
% \mnote{Djoko: I changed this to one Proposition.  - Mike},
relating the indicator wavenumber $\ktheta$ to $\keta$.
The interesting cases are where there are two spectra for the tracer, which in 2D is expected when the injection wavenumbers for tracer are below those for the fluid.
% Instead of the spectra, we assume the corresponding power laws, and that the bulk of the tracer is in the cascade ranges.

%In analogy with $\keta$ in 2D NSE, and assuming that at least part of the tracer cascade takes place in the enstrophy cascade range, we denote by $\kbeta$ the wavenumber beyond which diffusion should dominate the tracer, so
%\begin{equation*}
%   \kbeta=\left(\frac{\eta}{\mu^3}\right)^{1/6} \;.
%\end{equation*}

\subsection{Large Schmidt number}

\begin{figure}[h]
\psfrag{k0}{$\kz$}
\psfrag{k1}{$\kk^{-1}$}
\psfrag{k53}{$\kk^{-5/3}$}
\psfrag{ds}{$\keta$}
\psfrag{ds2}{$\kbeta$}
% \psfrag{q1}{$\kgdn$}
  \psfrag{q2}{$\kgup$}
  \psfrag{f1}{$\kfdn$}
  \psfrag{f2}{$\kfup$}
  \centerline{\includegraphics[scale=.7]{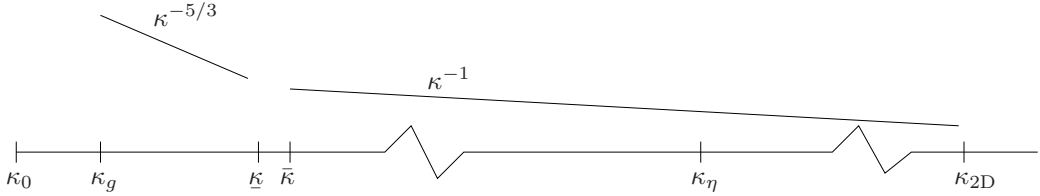}}
\caption{Expected tracer spectra for the case of inverse cascade with large Schmidt number.}
\label{specfig2}
\end{figure}
For large Schmidt number $\mu/\nu$, there is a range $[\keta,\kbeta]$ where the tracer is advected by a viscous fluid flow.
According to the classical theory \cite[pp.~367--9]{Vallis}, here we expect a $\kappa^{-1}$ tracer spectrum:
First, the time scale for this range is determined by substituting $\keta$ into \eqref{totem}, which gives
\begin{equation}\label{totem2D}
\tau_{\keta}=\eta^{-1/3}\;.
\end{equation}
One then sets $\tau_{\keta}$ equal to the diffusive time scale $(\mu\kk^2)^{-1}$ to find
\begin{equation}\label{kbetadef}
   \kbeta := \bigl({\eta}/{\mu^3}\bigr)^{1/6}
	= \Sc^{1/2}\keta\,.
\end{equation}
Using \eqref{totem2D} in \eqref{spec3}, and solving for $\cT(\kk)$
gives the same tracer spectrum as for $\kfup \le \kk \le \keta$,
\begin{equation}\label{spec55-2d}
   \cT(\kk)\sim \chi\eta^{-1/3}\kk^{-1}
   \quad \text{for} \quad \kfup \le \kk \le \kbeta \;.
\end{equation}

% Though larger Schmidt numbers make the objective range for the tracer cascade wider, the estimate for $\ktheta$ is compensated by a helpful $\Sc$ term.

Assuming power laws corresponding to the tracer spectra, we relate $\ktheta$ to $\keta$ and show that asymptotically $\ktheta\sim\kbeta$ for large $\Pr$:

\begin{thm}\label{2DPrext}
Suppose that $\kgup < \kfdn$ holds along with
\begin{equation}\label{bigksig}
   \ksig^2 \sim \keta^2/\ln({\keta}/{\kfup})\,,
\end{equation}
\begin{equation}\label{bulk22}
   \lgl | \theta|^2\rgl \sim \lgl | \theta_{\kg,\kfdn}|^2\rgl+\lgl | \theta_{\kfup,\kbeta}|^2\rgl
\end{equation}
and
\begin{subnumcases}
{\bft_{\kk,2\kk}\sim} \chi \epsilon^{-1/3}\kk^{-2/3}
     & for  $\kg \le \kk \le \kfdn$ \label{law33}\\
                                  \chi \eta^{-1/3}
     & for $\kfup \le \kk \lesssim \kbeta$  \label{law44} \;.
\end{subnumcases}
We then have
\begin{equation}\label{ab}
  \ktheta^2 \sim \frac{1}{a+b}\;,
\end{equation}
where
\[
  a=\kbeta^{-4/3}\Pr^{-1/3}\bigl(\kgup^{-2/3}-\kfdn^{-2/3}\bigr)
	\ln(\keta/\kfup)^{-1/3}  \quad \text{and} \quad b=\kbeta^{-2}\ln(\kbeta/\kfup).
\]

If, moreover,
\begin{equation}\label{sets1}
    \kg \sim \kz \aand \kfdn\sim\kfup
%  &\zeta := \kfup/\kz \ll G \label{sets2}
\end{equation}
along with, for any $r\in[4/3,2)$,
\begin{equation}\label{q:bigPr}
   \Pr \gtrsim (G\,\kfup/\kz)^{(3r-4)/(12-6r)}
   \aand
   \kfup/\kz \le (G\,(\ln G)^{1/2}/\ex)^{1/5},
\end{equation}
we have
\begin{equation}\label{bigPrkth}
  \ktheta^2 \sim \kbeta^r \kz^{2-r}/\ln(\kbeta/\kfup).
\end{equation}
\end{thm}

\begin{proof}
First we estimate over the inverse cascade as follows
\begin{align*}
   \bft_{\kgup,\kfdn}
	&\sim \frac{\chi}{\mu}\Bigl(\frac{\mu^3}{\epsilon}\Bigr)^{1/3}\bigl(\kgup^{-2/3}-\kfdn^{-2/3}\bigr)
	& &\text{by (\ref{sum1}), (\ref{kbetadef})}\\
	&= \frac{\chi}{\mu}\kbeta^{-2} \ksig^{2/3}\bigl(\kgup^{-2/3}-\kfdn^{-2/3}\bigr) & &\\
	&\sim \frac{\chi}{\mu}\kbeta^{-2}\keta^{2/3}\bigl(\kgup^{-2/3}-\kfdn^{-2/3}\bigr) \ln(\keta/\kfup)^{-1/3} & &\text{by \eqref{bigksig}}\\
	&= \frac{\chi}{\mu}\kbeta^{-4/3}\Pr^{-1/3}\bigl(\kgup^{-2/3}-\kfdn^{-2/3}\bigr) \ln(\keta/\kfup)^{-1/3} & &\text{by \eqref{kbetadef}.}
\end{align*}
Then over the range beyond $\kfup$ we find%
% \mnote{Mike: I removed the second $\sim$.}
\begin{equation}\label{q:inersum}
   \bft_{\kfup,\kbeta} \sim \frac{\chi}{\mu\kbeta^2} \ln(\kbeta/\kfup).
%  \sim \frac{\chi}{\mu\keta^2}\ln(\keta\Pr^{1/2}/\kfup) \;.
\end{equation}
It follows from \eqref{bulk22} that
\begin{equation}\label{gimme}
\bft_{\kz,\infty} \sim  \bft_{\kgup,\kfdn} +  \bft_{\kfup,\kbeta} %+   \bft_{\keta,\infty}
 %\\ &
 \sim  \frac{\chi}{\mu} \left(a +  b \right)
\end{equation}
and hence
\begin{equation} \label{gimme2}
\ktheta^2 =\frac{({\mu}/{L^2})\, \lgl |\nabla \theta|^2 \rgl}{({\mu}/{L^2})\, \lgl | \theta|^2 \rgl}
 = \frac{\chi}{\mu \bft_{\kz,\infty}} \sim \frac{1}{a+b} \;.
\end{equation}
From \eqref{ketasharp} we have
% for $G \gtrsim (\kfup/\kz)^2$
\begin{equation}\label{ketasharpbar}
   \left(\frac{\kz}{\kfup}\right)^{5/4} \frac{G^{1/4}}{(\ln G)^{1/4}} \lesssim \frac{\keta}{\kfup}
	\lesssim \left(\frac{\kz}{\kfup}\right)^{3/4} G^{1/4}(\ln G)^{1/8} \;.
\end{equation}

For the second part of the theorem, we seek to majorise $a$ as
\begin{equation} \label{wantPr}
   a \lesssim \kbeta^{-r} \kz^{r-2}\ln (\kbeta/\kfup),
\end{equation}
which is equivalent to
\begin{equation*}
   \left(\frac{\kbeta}{\kz}\right)^{r-4/3}\left[\left(\frac{\kg}{\kz}\right)^{-2/3}-\left(\frac{\kfdn}{\kz}\right)^{-2/3} \right]
	\lesssim \Pr^{1/3}\left( \ln \frac{\keta}{\kfup}\right)^{1/3}\ln \left(\frac{\keta}{\kfup}\Pr^{1/2}\right).
\end{equation*}
Since $\kfdn>\kg$ and $\kg\sim\kz$ (but $\kg>\kz$), we have by \eqref{sets1}
\[
   (\kg/\kz)^{-2/3}-(\kfdn/\kz)^{-2/3}\sim(\kg/\kz)^{-2/3}\sim1,
\]
so the last inequality is in turn equivalent to
\begin{equation}\label{q:aux00}
   ({\keta/\kz})^{r-4/3}
          \lesssim \Pr^{1-r/2}\Bigl(\ln\frac\keta\kfup\Bigr)^{1/3} \Bigl(\ln\frac\keta\kfup + \ln\Pr\Bigr).
\end{equation}
From the upper bound in \eqref{ketasharpbar} we have, with $\zeta:=\kfup/\kz$,
\begin{equation}
   \keta/\kz \lesssim (\zeta G)^{1/4}(\ln G)^{1/8}.
\end{equation}
% where we have used $\zeta\ll G$ from \eqref{sets2} for the second inequality.\mnote{Mike: if we want $G^{1/4}$, we need $\zeta\sim1$.}
Using this to bound the left-hand side of \eqref{q:aux00}, we have
\begin{equation}\label{q:aux01}
   (\keta/\kz)^{r-4/3} \lesssim (\zeta G)^{r/4-1/3}(\ln G)^{r/8-1/6}.
\end{equation}
Now the lower bound in \eqref{ketasharpbar} is
\[
   \zeta^{-5/4}(G/\ln G)^{1/4} \lesssim \keta/\kfup,
\]
which we then apply to the right-hand side of \eqref{q:aux00} to obtain
\begin{equation*}\begin{aligned}
   \Bigl(\ln\frac\keta\kfup\Bigr)^{1/3}\Bigl(\ln\frac\keta\kfup+\ln\Pr\Bigr)
	&\gtrsim \ln\Bigl(\frac{G}{\zeta^5\ln G}\Bigr)^{1/3}\Bigl[\ln\Bigl(\frac{G}{\zeta^5\ln G}\Bigr)+\ln\Pr\Bigr]\\
	&\gtrsim \ln\Bigl(\frac{G}{\zeta^5\ln G}\Bigr)^{4/3}.
\end{aligned}\end{equation*}
Putting this together with \eqref{q:aux01}, we find that \eqref{q:aux00} is implied by
\begin{equation}\label{q:aux02}
   (\zeta G\,(\ln G)^{1/2})^{r/4-1/3} \lesssim \Pr^{1-r/2}\ln\Bigl(\frac{G}{\zeta^5\ln G}\Bigr)^{4/3}.
\end{equation}
Now for $G\ge1$ we have
\[
   \bigl(G\,(\ln G)^{1/2}\bigr)^{1/2} \le G/\ln G,
\]
so assuming this and writing $\gamma:=G\,(\ln G)^{1/2}$, \eqref{q:aux02} is implied by
\begin{equation}
   (\zeta \gamma)^{r/4-1/3} \lesssim \Pr^{1-r/2}\ln\bigl(\zeta^{-5}\gamma\bigr)^{4/3}.
\end{equation}
With the extra assumption $\gamma\zeta^{-5}\ge\ex$,
we see that \eqref{q:bigPr} implies \eqref{bigPrkth}.
% we arrive at the conclusion
% \begin{equation}
%    \zeta G\ln^{1/2}G \le \Pr^{(12-6r)/(3r-4)}.
% \end{equation}
\end{proof}

% \bigskip\hbox to\hsize{\qquad\hrulefill\qquad}\medskip

% \subsection{Small Schmidt numbers}
%
% For $\Sc \ll 1$ in 2D we have $\ktheta^2 \sim 1/a$, which by \eqref{ketasharp} would achieve
% $\ktheta \sim \kbeta^{2/3}$ if $G \sim \exp(\Sc^{-1})$, up to a logarithm.

\subsection{Moderate Schmidt number}\label{modsubsection}
%$\nu\sim\mu$, $\kgup  <  \underline\kappa_f$} \hfill \break

For moderate Schmidt numbers, i.e. $\nu/\mu \sim 1$, we expect $\kbeta \sim \keta$.
In the simplest case, with $\kgup=\kfup$ and $\ksig^2 \sim \keta^2/\ln({\keta}/{\kfup})$, the tracer cascade occurs in the enstrophy cascade range, viz.
\begin{align}
   &\lgl|\theta|^2\rgl \sim \lgl|\theta_{\kfup,\keta}|^2\rgl,\label{q:sc1-thtsp}\\
   &\bft_{\kk,2\kk} \sim \chi\eta^{-1/3}
	\qquad\textrm{for } \kfup\le\kk\lesssim\keta.\label{q:sc1-ensc1}
\end{align}
We then have
\begin{align*}
   \bft_{\kfup,\keta} &\sim \chi\eta^{-1/3}\ln(\keta/\kfup) & &\textrm{by \eqref{q:sc1-ensc1} and \eqref{sum2}}\\
	&= \frac{\chi}{\mu}\Bigl(\frac{\mu^3}{\eta}\Bigr)^{1/3}\ln(\keta/\kfup)
	& &= \frac{\chi}{\mu\keta^2}\ln(\keta/\kfup),
\end{align*}
which by \eqref{q:sc1-thtsp} implies
\begin{equation}
   \bft_{\kz,\infty} \sim \bft_{\kg,\keta} \sim \frac{\chi}{\mu\keta^2}\ln(\keta/\kfup).
\end{equation}
Thus, $\ktheta\sim\keta$ up to logarithm,
\begin{equation}
   \ktheta^2 = \frac{\lgl|\nabla\theta|^2\rgl}{\lgl|\theta|^2\rgl}
	= \frac{\chi}{\mu\bft_{\kz,\infty}}
	\sim \keta^2/\ln(\keta/\kfup)
	\sim \kbeta^2/\ln(\kbeta/\kfup).
\end{equation}

If the energy injection scale is small compared to the tracer injection scale, i.e.\ $\kgup\ll\kfdn$, we expect to have two tracer cascade ranges (both downscale).
In the gap between $\kgup$ and $\kfdn$, the energy spectrum is expected to take the form $\cE(\kappa)\sim \epsilon^{2/3}\kappa^{-5/3}$, so that the tracer spectrum should be $\cT(\kappa)\sim \chi \epsilon^{-1/3} \kappa^{-5/3}$.
On the other hand, in the enstrophy cascade range $\kfup \le \kappa \lesssim \keta$, we expect $\cE(\kappa)\sim \eta^{2/3}\kappa^{-3}$, so that $\cT(\kappa)\sim \chi \eta^{-1/3} \kappa^{-1}$.
This case is virtually identical to that treated in Theorem~\ref{2DPrext}, putting $\kbeta\sim\keta$ since $\Sc=\nu/\mu\sim1$.

\begin{figure}[h]
\psfrag{k0}{$\kz$}
\psfrag{k1}{$\kk^{-1}$}
\psfrag{k53}{$\kk^{-5/3}$}
\psfrag{ds}{$\keta \sim\kbeta$}
% \psfrag{q1}{$\kgdn$}
  \psfrag{q2}{$\kgup$}
  \psfrag{f1}{$\kfdn$}
  \psfrag{f2}{$\kfup$}
  \centerline{\includegraphics[scale=0.7]{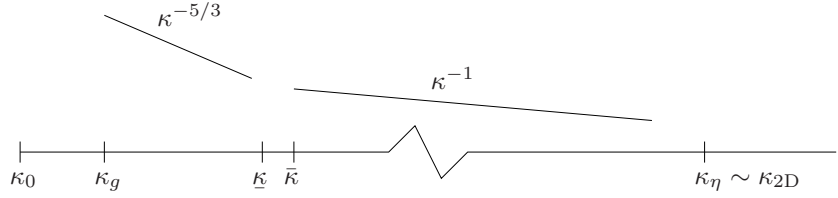}}
\caption{Expected tracer spectra for case of inverse cascade with moderate Schmidt number.}
\label{specfig}
\end{figure}

Note that by Proposition \ref{2Dconnect}  the condition \eqref{bigksig} could be replaced by the more natural, but stronger assumptions
\begin{align}
   &\bfe_{\kk,2\kk}\sim \eta^{2/3}\kk^{-2}
\quad \text{for} \quad  \kfup \le \kk \lesssim \keta,  \\ %\label{law2}
   &\lgl\|u_{\kz,\kfup}\|^2\rgl \lesssim \lgl\| u_{\kfup,\infty}\|^2\rgl \\ %\label{bulk1}
   &4\kfup \le \keta
%\end{subnumcases}
\end{align}
which are consistent with the discrete tracer spectrum \eqref{law44}.

The main influence on the $\ktheta$ estimate in Theorem~\ref{2DPrext} (with $\kbeta\sim\keta$) is the assumption
\begin{equation}\label{mainspec}
   \bft_{\kk,2\kk} \sim \chi \eta^{1/3} \qquad \text{for} \quad \kfup \le \kk \lesssim  \keta \sim \kbeta \;.
\end{equation}
The following is partial converse.
\begin{thm}
If $\ktheta \gtrsim \kbeta$, then $\bft_{\kk,2\kk} \lesssim \chi\eta^{1/3}$ \;.
\end{thm}
% The proof is similar to one in \cite{DFJ3} showing that $\ksig \gtrsim \keta \Rightarrow \bfe_{\kk,2\kk} \lesssim \eta^{2/3}\kk^{-2}$.
% \mnote{Djoko: I can't find this in \cite{DFJ3}, so I must have seen it elsewhere.}
\begin{proof}
We can rewrite the assumption as
\[
   \ktheta^2 =\frac{\lgl |\nabla \theta|^2 \rgl}{\lgl |\theta |^2\rgl}
   \gtrsim\left( \frac{\eta}{\mu^3}\right)^{1/3} = \kbeta^2 \;,
\]
or as
\[
   \frac{\mu}{L^2}\lgl |\nabla \theta|^2 \rgl  \gtrsim \frac{\eta^{1/3}}{L^2}
   \lgl |\theta |^2\rgl,
\]
so that
\[
   \chi \eta^{-1/3} \gtrsim \frac{1}{L^2}  \lgl |\theta |^2\rgl
	\ge  \frac{1}{L^2}  \lgl |\theta_{\kk,2\kk} |^2\rgl = \bft_{\kk,2\kk}\;.
\]
\end{proof}

% \bigskip\hbox to\hsize{\qquad\hrulefill\qquad}\medskip
% \pagebreak

Theorem \ref{2DPrext} (with $\kbeta\sim\keta$) imposes a restriction on the ranges of the forcing terms and the Grashof number.
The indicator $\ktheta$ would achieve its maximum value, $\keta\sim\kbeta$ (up to a log), if one could choose $\kgup$, $\kfdn$, and $\kfup$ in such a way that $a\sim b$.
To investigate this, we seek $r$ such that
\begin{equation}\label{want1}
a \le c_0 \keta^{-r}\kz^{r-2}\ln\frac{\keta}{\kfup}
	\quad\text{where}\quad  \frac{4}{3} \le r \le 2,
\end{equation}
for some $c_0$, which is equivalent to
\begin{equation}\label{want2}
\left(\frac{\keta}{\kz}\right)^{r-4/3}\left[\left(\frac{\kg}{\kz}\right)^{-2/3}-
          \left(\frac{\kfdn}{\kz}\right)^{-2/3} \right] \le c_0 \left( \ln \frac{\keta}{\kfup}\right)^{4/3}.
\end{equation}

We now derive a sufficient condition for~\eqref{want2}.
Rewriting the lower bound in \eqref{ketasharpbar} as
\begin{equation}\label{q:aux3}
   c_1 \Bigl(\frac\keta\kz\Bigr)^{5/4} \frac{G^{1/4}}{(\ln G)^{1/4}}
	\le \frac{\keta}{\kfup},
\end{equation}
and the upper bound in \eqref{ketasharp} as
\begin{equation}\label{q:aux4}
   \frac\keta\kz \le c_2 \Bigl(\frac\kfup\kz\,G\,(\ln G)^{1/2}\Bigr)^{1/4},
\end{equation}
we use \eqref{q:aux3} on the right and \eqref{q:aux4} on the left in \eqref{want2},
% Using the lower bound in \eqref{ketasharpbar} on the right and the upper bound in \eqref{ketasharp} on the left in \eqref{want2},
we obtain, with $p:=(3r-4)/12\in[0,1/6]$, $\zeta:=\kfup/\kz$ and $c_3=c_0 c_1/c_2^p$,
% \begin{equation*}\label{want3p}
% \left[\frac{\kfup}{\kz}G(\ln G)^{1/2}\right]^p \left[\left(\frac{\kg}{\kz}\right)^{-2/3}-
%           \left(\frac{\kfdn}{\kz}\right)^{-2/3} \right]
%           \lesssim
%           \left\{\ln \left[\left(  \frac{\kz}{\kfup}\right)^5 \frac{G}{\ln G}\right]\right\}^{4/3},
% \end{equation*}
% or equivalently with $\zeta:=\kfup/\kz$ and an explicit constant,
\begin{align}
   &[\zeta G(\ln G)^{1/2}]^p\bigl(1-\zeta^{-2/3}\bigr)
	\le c_3\, \ln\bigl(\zeta^{-5}G/\ln G\bigr)^{4/3}\notag\\
   \Leftrightarrow\qquad
   &\frac1{c_3} \le \frac{\ln\bigl(\zeta^{-5}G/\ln G\bigr)^{4/3}}{[\zeta G(\ln G)^{1/2}]^p(1-\zeta^{-2/3})}.
\end{align}
Putting $G=\ex^\zeta$, this in turn is equivalent to
\begin{equation}\label{q:want3}
   \frac1{c_3} \le \frac{(\zeta-6\ln\zeta)^{4/3}}{\zeta^{3p/2}\ex^{p\zeta}(1-\zeta^{-2/3})}
	=: \varphi_p(\zeta).
\end{equation}

\begin{figure}[h]
\psfrag{zeta}{$\scriptstyle\zeta$}
\includegraphics[scale=0.5]{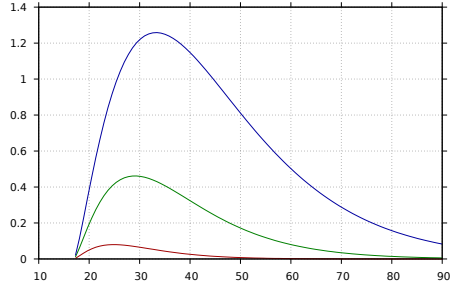}
\caption{From bottom to top: $\varphi_{1/6}$, $\varphi_{1/9}$ and $\varphi_{1/12}$.}
\label{vf1}
\end{figure}

In figure~\ref{vf1} we plot $\varphi_{1/6}$, $\varphi_{1/9}$ and $\varphi_{1/12}$ against $\zeta$.
It is clear that, at least for these values of $p$, there is a range of $\zeta=\kfup/\kz$ such that \eqref{q:want3}, and thus \eqref{want2}, is satisfied, {\em provided\/} that $c_3$ is sufficiently large.
(Since we are seeking a sufficient condition for \eqref{want2}, we can take $c_3$ smaller but not larger.)
While a good estimate for $c_3$ is not known, this plot suggests that even in the presence of a significant inverse cascade ($10\lesssim\zeta\lesssim20$) a wide tracer cascade range can be achieved,
\begin{equation}
   \ktheta^2 \sim \kbeta^r\kz^{2-r}/\ln\bigl(\kbeta/\kfup\bigr)
\end{equation}
with $r=2$, $16/9$ and $5/3$ for $p=1/6$, $1/9$ and $1/12$ respectively,
for large enough Grashof number ($G\sim\ex^\zeta$) to sustain turbulent fluid flow.

% \hbox to\hsize{\qquad\hrulefill\qquad}\medskip

\subsection{Effect of log corrected energy spectrum}

In order to enforce constant enstrophy flux Kraichnan \cite{Kr71} proposed a log correction to the energy spectrum in the inertial range for 2D turbulence
\begin{align*}
   \cE(\kappa)\sim \eta^{2/3}\kappa^{-3}(\log\kk/\kidn)^{-1/3}\;,
\end{align*}
which leads to a turnover time of
\begin{equation}\label{new_over}
\tau_\kk \sim \eta^{-1/3}(\log \kk/\kidn)^{-1/3} \;.
\end{equation}
This correction was shown in \cite{Ohkitani} to be consistent with an upper bound on the dimension of the global attractor in \cite{CFT89}.

If \eqref{new_over} is used in \eqref{spec3}, and the lower end of the inertial
range is $\kidn=\kfup$, the tracer spectrum takes the form
\begin{align*} \label{spec_log}
   \cT(\kk) \sim \chi \eta^{-1/3} \kk^{-1}(\log \kk/\kfup)^{-1/3}\;, \quad \kfup \le \kk \le \keta
\end{align*}
and
\begin{align*}
\bft_{\kk,2\kk} \sim \int_\kk^{2\kk} \cT(s) \ ds \sim \chi\eta^{-1/3}\left[(\log 2\kk/\kfup)^{2/3}-(\log \kk/\kfup)^{2/3}\right]
\;.
\end{align*}
Summing as in Lemma \ref{sumlem}, the terms telescope, so that
\begin{equation}
   \bft_{\kfup,\keta} \sim \frac{\chi}{\mu} \left(\frac{\mu^3}{\eta}\right)^{1/3} (\ln \keta/\kfup)^{2/3}
	 \sim \frac{\chi}{\mu\keta^2} (\ln \keta/\kfup)^{2/3} \;.
\end{equation}
Using this instead of \eqref{q:inersum} in the proof of Theorem \ref{2DPrext} yields
\begin{align*}
   \ktheta^2 \sim \frac{1}{a+b'}\;, \quad \text{where} \quad b' =\keta^{-2} (\ln \keta/\kfup )^{2/3} \;.
\end{align*}
We now seek $r$ such that
\begin{align*}%\label{want1}
   a \lesssim \keta^{-r}\kz^{2-r}(\ln\frac{\keta}{\kfup})^{2/3}\;, \quad \text{where} \quad  \frac{4}{3} \le r \le 2\;,
\end{align*}
which is equivalent to the analog of \eqref{want2}
\begin{equation} \label{want22}
\left(\frac{\keta}{\kz}\right)^{r-4/3}\left[\left(\frac{\kg}{\kz}\right)^{-2/3}-
          \left(\frac{\kfdn}{\kz}\right)^{-2/3} \right] \lesssim \ln \frac{\keta}{\kfup} \;,
\end{equation}
the only change being the power of the log on the right.

Proceeding as before, with $\zeta=\kfup/\kz$ and $p=(3r-4)/12$, and putting $G=\ex^\zeta$, this is implied by
\begin{equation}\label{q:want4}
   \frac1{c_4} \le \frac{\zeta-6\ln\zeta}{\zeta^{3p/2}\ex^{p\zeta}(1-\zeta^{-2/3})}
	=: \tilde\varphi_p(\zeta).
\end{equation}

\begin{figure}[h]
\psfrag{zeta}{$\scriptstyle\zeta$}
\includegraphics[scale=0.5]{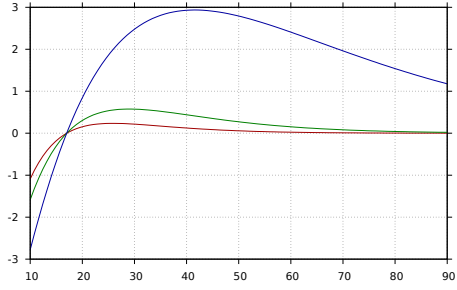}
\caption{From bottom to top: $\tilde\varphi_{1/9}$, $\tilde\varphi_{1/12}$ and $\tilde\varphi_{1/24}$.}
\label{vf2}
\end{figure}

\noindent
In figure~\ref{vf2} we plot $\tilde\varphi_{1/9}$, $\tilde\varphi_{1/12}$ and $\tilde\varphi_{1/24}$ against $\zeta$.
Again, we need $c_4$ sufficiently large for \eqref{q:want4} to hold.

\subsection{Tracer injection scales below energy injection scales}

In case the injection scales are reversed, so that $\kfup < \kgup$, then the analysis for both moderate and large Schmidt number proceeds as before, except the term $a$ is dropped in both cases, so the conclusion is that $\ktheta \sim \kbeta$ (up to a log).

% ===========================================================================

\section{3D Case}\label{3DTracer}

The large Schmidt number case is also interesting in 3D, as then we expect two ranges with distinct tracer spectra \cite[p.~368]{Vallis}:
For $\kappa\in(\kfup,\keps)$, we have the classical spectrum $\cT(\kappa)\sim\kappa^{-5/3}$.
For $\kappa$ beyond $\keps$,  substituting $\kk=\keps$ in \eqref{totem} gives a turnover time of
\begin{equation}\label{totem3D}
   \tau_{\keps}=(\nu/\epsilon)^{1/2} \;.
\end{equation}
Putting this equal to the diffusive time scale $(\mu\kk^2)^{-1}$ then yields
\begin{equation}\label{q:kbeta3d}
   \kbetap = \Bigl(\frac{\epsilon}{\nu\mu^2}\Bigr)^{1/4}
	= \Pr^{1/2}\keps
\end{equation}
the wavenumber where diffusion becomes important.
Using \eqref{totem3D} in \eqref{spec3}, and solving for $\cT(\kk)$ gives
\begin{equation}\label{spec55-3d}
   \cT(\kk)\sim \chi\left(\frac{\nu}{\epsilon}\right)^{1/2}\kk^{-1}
	\quad \text{for} \quad \keps \le \kk \le \kbetap.
\end{equation}

We have the following analogue of Theorem~\ref{2DPrext}:

\begin{thm}\label{3DbigPr}
Suppose that \eqref{ktaubig} holds along with $4\kg \le \keps$, $\Pr > 2$,
\begin{equation}
   \lgl | \theta |^2\rgl \sim \lgl | \theta_{\kgup,\kbetap}|^2\rgl,
\end{equation}
and
\begin{align*}
   \bft_{\kk,2\kk}\sim
\begin{cases}
\chi \epsilon^{-1/3}\kk^{-2/3}
     & \text{for} \quad  \kgup \le \kk \le \keps \\
                                  \chi ({\nu}/{\epsilon})^{1/2}
     & \text{for} \quad  \keps \le \kk \lesssim \kbetap
     \;.
\end{cases}
\end{align*}
We then have
\begin{equation}\label{q:ktht3d}
   \ktheta^2 \sim \frac{1}{a+b}
\end{equation}
where
\begin{equation}
   a = \kbetap^{-4/3}\Pr^{-1/3}\bigl(\kgup^{-2/3}-\keps^{-2/3}\bigr)
	\aand
	b = \kbetap^{-2} \ln(\Pr).
\end{equation}
If, moreover, $\kg\sim\kz$ and $\kfdn\sim\kfup$, along with
\begin{equation}
  \Pr \gtrsim G^{(3r-4)/(8-4r)},
\end{equation}
then
\begin{equation}\label{bigPrkth3D}
  \ktheta^2   \sim   \kbetap^r \kz^{2-r}/\ln(\kbetap/\keps)
	\quad \text{for} \quad  4/3 \le r < 2.
\end{equation}
\end{thm}

\begin{proof}
As in the 2D case, we first compute
\begin{align*}
   \bft_{\kgup,\keps} &\sim\frac{\chi}{\mu}\Bigl(\frac{\mu^3}{\epsilon}\Bigr)^{1/3}\bigl(\kgup^{-2/3}-\keps^{-2/3}\bigr)
	= \frac{\chi}{\mu}\kbetap^{-4/3}\Pr^{-1/3}\bigl(\kgup^{-2/3}-\keps^{-2/3}\bigr)\\
\bft_{\keps,\kbetap} &\sim \frac{\chi}{\mu} \Bigl(\frac{\nu\mu^2}{\epsilon}\Bigr)^{1/2} \ln(\kbetap/\keps)
	\sim \frac{\chi}{\mu}\kbetap^{-2} \ln(\Pr).
\end{align*}
By hypothesis, $\lgl|\theta|^2\rgl\sim\bft_{\kg,\keps}+\bft_{\keps,\kbeta}$, giving us \eqref{q:ktht3d}
\begin{equation}
   \ktheta^2 = \frac{\chi}{\mu\bft_{\kz,\infty}}
	= \frac1{a+b}.
\end{equation}

For the second part of the theorem, we note that
\[
   a \lesssim \kbeta^{-r}\kz^{r-2}\ln(\Pr)
\]
is equivalent to
\begin{align*}
   &(\kbetap/\kz)^{r-4/3}\bigl[(\kz/\kg)^{2/3}-(\kz/\kfup)^{2/3}\bigr]
	\lesssim \Pr^{1/3}\log(\Pr)\\
   \Leftrightarrow\qquad
   &(\keps/\kz)^{r-4/3}\lesssim \Pr^{1-r/2}\log(\Pr).
\end{align*}
Arguing as in the 2D case, we bound the lhs by the upper bound in \eqref{sharp3D} and using $\log(\Pr)>1$ on the rhs, this is implied by
\[
   G^{(3r-4)/8} \lesssim \Pr^{1-r/2},
\]
which gives us \eqref{bigPrkth3D}.
\end{proof}

\begin{obs}\normalfont
The decay rate of the energy spectrum in the $(\kg,\keps)$-inertial range is not crucial here.
It is the prefactor in the tracer spectrum that produces the helpful Schmidt number effect in the estimate in Theorem \ref{3DbigPr}.
In fact, we would achieve the same estimate for $\ktheta$ if we consider a dimensionally correct energy spectrum with a different decay rate
\[
   \cE_{\text{3D}}(\kk)\sim \epsilon^{2/3} \kz^{p-5/3}\kk^{-p}  \quad \text{ for any} \quad p \in (1,3).
\]
Note that this would violate Kolmogorov's assumption that $\cE_{\text{3D}}$ depend on only $\epsilon$ and $\kk$, as it would now also depend on $L$.
Nevertheless, an energy spectrum of this form would result in a tracer spectrum \cite[(8.94)]{Vallis}
\[
   \cT(\kk)\sim\chi\epsilon^{-1/3}\kz^q\kk^{q'-1}
	\qquad\text{with } q=(p-3)/2 \text{ and } q'=(5-3p)/6,
\]
corresponding to a discrete dyadic tracer spectrum
\[
   \bft_{\kk,2\kk}\sim \chi\epsilon^{-1/3}\kz^{q'}\kk^q \quad \text{for} \quad  \kg \le \kk \le \keps.
\]
Assuming again, that $\kz \sim \kg \ll \keps$, we have
\begin{align*}
   \bft_{\kg,\keps} \sim \chi \epsilon^{-1/3}\kz^{q'}\left(\kg^q-\keps^q\right)
	&\sim \chi\epsilon^{-1/3} \kz^{-2/3}\\
	&= \frac{\chi}{\mu}\Bigl(\frac{\nu\mu^2}{\epsilon}\Bigr)^{1/3}\left(\frac{\mu}{\nu}\right)^{1/3} \kz^{-2/3} \\
	&= \frac{\chi}{\mu}\kbetap^{-4/3}\Sc^{-1/3} \kz^{-2/3}\;.
\end{align*}
The rest of the estimate for $\ktheta$ follows as in the proof of Theorem \ref{3DbigPr}.
\end{obs}

\subsection{Moderate Schmidt number case, 3D}

If in 3D, $\Sc \sim 1$, we have just the single, steeper tracer spectrum, and $\ktheta^2 \sim 1/a$ with $a$ as in Theorem \ref{3DbigPr}.
This can be expressed as $\ktheta\sim \kbetap^{2/3}\kz^{1/3}$, which gives the same fractional power for the tracer cascade range width as for the energy cascade in Proposition \ref{3Dconnect}.

\nocite{davidson:t}

% \mnote{References need to be revised.}
% \bibliographystyle{siam}
% \bibliography{Tracer,all}

\end{document}